\RequirePackage{etoolbox}
\csdef{input@path}{%
 {sty/}
 {img/}
}%
\csgdef{bibdir}{bib/}

\documentclass[ba]{imsart}
\pubyear{0000}
\volume{00}
\issue{0}
\doi{0000}
\firstpage{1}
\lastpage{1}

\usepackage{amsmath}
\usepackage{natbib}
\usepackage[colorlinks,citecolor=blue,urlcolor=blue,filecolor=blue,backref=page]{hyperref}
\usepackage{graphicx}

\usepackage{amsmath}   
\usepackage{amsfonts}
\usepackage{subfig}
\usepackage{latexsym}
\usepackage{upgreek}
\usepackage{amssymb}

\usepackage{amsthm}

\usepackage{fancyhdr}
\usepackage{mathrsfs}
\usepackage{framed}
\usepackage{bm}
\usepackage{perpage} 
\usepackage{enumerate}
\usepackage{bbm}
\usepackage{dsfont}
\usepackage{comment}
\usepackage[utf8]{inputenc}
\usepackage{color}
\definecolor{darkgreen}{rgb}{0,0.7,0}

\definecolor{darkblue}{rgb}{0,0,0.7}

\renewcommand{\div}{\text{div}}

\newcommand{\argmin}{\mbox{argmin}}

\newcommand{\R}{\mathbb{R}}

\newcommand{\M}{\mathcal{M}}
\newcommand{\p}{\mathcal{P}}

\newcommand{\ptwo}{\mathcal{P}_2}

\newcommand{\dkl}{D_{\mbox {\tiny{\rm KL}}}}
\newcommand{\dchi}{D_{\mbox {\tiny{\rm $\chi^2$}}}}

\newcommand{\jkl}{J_{\mbox {\tiny{\rm KL}}}}
\newcommand{\jchi}{J_{\mbox {\tiny{\rm $\chi^2$}}}}
\newcommand{\fchi}{F_{\mbox {\tiny{\rm $\chi^2$}}}}
\newcommand{\fkl}{F_{\mbox {\tiny{\rm KL}}}}

\newcommand{\W}{\mathcal{W}}

\newcommand{\envelope}{(\raisebox{-.5pt}{\scalebox{1.45}{\Letter}}\kern-1.7pt)}

\newcommand{\veps}{\varepsilon}

\definecolor{mygreen}{rgb}{0.1,0.75,0.2}

\newcommand{\nc}{\normalcolor}

\usepackage[normalem]{ulem}

\newcommand{\pX}[1]{\frac{\partial}{\partial x_{#1}}}

            \newtheorem{thm}{Theorem}[section]
          \newtheorem{prop}[thm]{Proposition}

          \newtheorem{remark}[thm]{Remark}

          \newtheorem{defn}[thm]{Definition}
          
          \newtheorem{example}{Example}

          \newtheorem{rem}[thm]{Remark}

\startlocaldefs
\numberwithin{equation}{section}
\theoremstyle{plain}

\endlocaldefs

\begin{document}

\begin{frontmatter}
\title{The Bayesian update: variational formulations and gradient flows\thanksref{T1}}
\runtitle{The Bayesian update: variational formulations and gradient flows}

\begin{aug}
\author{\fnms{Nicolas} \snm{Garcia Trillos}\thanksref{addr1}\ead[label=e1]{garciatrillo@wisc.edu}},
\and
\author{\fnms{Daniel} \snm{Sanz-Alonso}\thanksref{addr2}\ead[label=e2]{sanzalonso@uchicago.edu}}

\runauthor{}

\address[addr1]{Department of Statistics, University of Wisconsin Madison 
    \printead{e1} 
}

\address[addr2]{Department of Statistics, University of Chicago
    \printead{e2}
}

\end{aug}

\begin{abstract}
          The Bayesian update can be viewed as a variational problem by characterizing the posterior as the minimizer of a functional. The variational viewpoint is far from new and is at the heart of popular methods for posterior approximation. However, some  of its consequences seem largely unexplored. We focus on the following one: defining the posterior as the minimizer of a functional gives a natural path towards the posterior  by moving in the direction of steepest descent of the functional. This idea is made precise through the theory of gradient flows, allowing to bring new tools to the study of Bayesian models and algorithms. Since the posterior may be characterized as the minimizer of different functionals, several variational formulations may be considered. We study three of them and their three associated gradient flows. We show that, in all cases, the rate of convergence of the flows to the posterior can be bounded by the geodesic convexity of the functional to be minimized. Each gradient flow naturally suggests a nonlinear diffusion with the posterior as invariant distribution. These diffusions may be discretized to build proposals for Markov chain Monte Carlo (MCMC) algorithms. By construction, the diffusions are guaranteed to satisfy a certain optimality condition, and rates of convergence are given by the convexity of the functionals. We use this observation to propose a criterion for the choice of metric in Riemannian MCMC methods.
\end{abstract}

\begin{keyword}[class=MSC]
\kwd[62C10]{}
\kwd{62F15}
\kwd{49N99}
\end{keyword}

\begin{keyword}
\kwd{Gradient flows, Wasserstein Space, Convexity, Riemannian MCMC}
\end{keyword}

\end{frontmatter}

   \section{Introduction}\label{intro}
 In this paper we revisit the old idea of viewing the posterior as the minimizer of an energy functional.  The use of variational formulations of Bayes rule seems to have been largely focused on one of its methodological benefits: restricting the minimization to a subclass of measures is the backbone of variational Bayes methods for posterior approximation. Our aim is to bring attention to two other theoretical and methodological benefits, and to study in some detail one of these: namely, that each variational formulation suggests a natural path, defined by a gradient flow, towards the posterior. We use this observation to propose a criterion for the choice of metric in Riemannian MCMC methods. 

Let us recall informally a variational formulation of Bayes rule. Given a prior $p(u)$ on an unknown parameter $u$ and a likelihood function $L(y|u),$ the posterior $p(u|y) \propto L(y|u) p(u)$ can be characterized \cite{duke} as the distribution $q^*(u)$ that minimizes 
\begin{equation}\label{jklformulaini}
\jkl\bigl(q(u)\bigr) = \dkl \bigl(q(u) \| p(u)\bigr) - \int \log L(y|u) q(u) du.
\end{equation}
Indeed, minimizing $\jkl\bigl( q(u) \bigr)$ is equivalent to minimizing the Kullback-Leibler divergence $\dkl(q(u)\| p(u|y)),$ and so clearly the minimizer is $q^*(u) = p(u|y).$
Other variational formulations may be considered by minimizing, for instance, other divergences rather than Kullback-Leibler. In this paper we consider three variational formulations of the Bayesian update. The first two characterize the posterior \emph{measure}  as the minimizer of functionals $\jkl$ or $\jchi$ constructed by penalizing deviations from the prior measure in Kullback-Leibler or $\chi^2$ divergence ---definitions of these functionals and divergences are given in equations \eqref{functionalkl}, \eqref{functionalchi}, \eqref{kldefinition}, \eqref{chidefinition}. The third one characterizes the posterior \emph{density} as the minimizer of the Dirichlet energy $D^\mu$ ---see \eqref{dirichletenergy}. 

Why is it useful to view the posterior $p(u|y)$ as the minimizer of an energy? We list below three advantages of this viewpoint, the third of which will be the focus of our paper.
\begin{enumerate}
\item The variational formulation provides a natural way to approximate the posterior by restricting the minimization problem to distributions $q(u)$ satisfying some computationally desirable property. 
For instance, variational Bayes methods often restrict the minimization to $q(u)$ with product structure \cite{attias1999variational}, \cite{wainwright2008graphical}, \cite{fox2012tutorial}.  A similar idea is studied in \cite{pinski2015kullback}, where $q(u)$ is restricted to a class of Gaussian distributions. An iterative variational procedure that progressively improves the posterior approximation by enriching the family of distributions was introduced in \cite{guo2016boosting}.
\item If the prior $p_{\varepsilon_n}(u)$ or the likelihood $L_{\varepsilon_n}(y|u)$ depend on a parameter $\epsilon_n,$ then the variational formulation allows to show large $n$ convergence of posteriors $p_{\varepsilon_n}(u|y)$ by establishing the $\Gamma$-convergence of the associated energies. This method of proof has been employed by the authors in \cite{trillos2017continuum}, \cite{trillos2017consistency} to analyze the large-data consistency of graph-based Bayesian semi-supervised learning. 
\item Each variational formulation gives a natural path, defined by a gradient flow, towards the posterior. These flows can be thought of as time-parameterized curves in the space of probability measures, converging in the large-time limit towards the posterior. 
\end{enumerate}

In this paper we study three gradient flows associated with the variational formulations defined by minimization of the functionals $\jkl$, $\jchi$, and $D^\mu.$ For intuition, we recall that a gradient flow in Euclidean space defines a curve whose tangent always points in the direction of steepest descent of a given function --see equation \eqref{ODE}. In the same fashion, a gradient flow in a more general metric space can be thought of as a curve on said space that always points in the direction of steepest descent of a given functional \cite{ambrosio2008gradient}. In Euclidean space the direction of steepest descent is naturally defined as that in which an \emph{Euclidean} infinitesimal increment leads to the largest decrease on the value of the function. In a general metric space the direction of steepest descent is the one in which an infinitesimal increment \emph{defined in terms of the distance}  leads to the largest decrease on the value of the functional.
In this paper we study: 
\begin{enumerate}
\item[(i)]  The gradient flows defined by $\jkl$ and $\jchi$ in the space of probability measures with finite second moments endowed with the Wasserstein distance---definitions are given in \eqref{defwasserstein}. By construction, these flows give curves of probability measures that evolve following the direction of steepest descent of $\jkl$ and $\jchi$ in Wasserstein distance, converging to the posterior measure in the large-time limit.
\item[(ii)] The gradient flow defined by the Dirichlet energy $D^\mu$ in the space of square integrable densities endowed with the $L^2$ distance. By construction, this flow gives a curve of densities in $L^2$ that evolves following the direction of steepest descent of $D^\mu$ in $L^2$ distance, converging to the posterior density in the large-time limit. Interestingly, the curve of measures associated with these densities is the exact same as the curve defined by the $\jkl$ flow on Wasserstein space \cite{jordan1998variational}.
\end{enumerate} 

A question arises:  what is the rate of convergence of these flows to the posterior?  The answer is, to a large extent, provided by the theory of optimal transport and gradient flows \cite{ambrosio2008gradient}, \cite{villani2003topics}, \cite{villani2008optimal}, \cite{santambrogio2015optimal}. We will review and provide a unified account of these results in the main body of the paper, section \ref{sec:section3mainbody}. In the remainder of this introduction we discuss how rates of convergence may be studied in terms of convexity of functionals, and how these rates may be used as a guide for the choice of proposals for MCMC methods.

Rates of convergence of the flows hinge on the convexity level of each of the functionals $\jkl,$ $\jchi,$ and $D^\mu.$ Recalling the Euclidean case may be helpful: gradient descent on a highly convex function will lead to fast convergence to the minimizer. What is, however, a sensible notion of convexity for  functionals defined over measures or densities? Our presentation highlights that the notion of geodesic (or displacement) convexity \cite{mccann1997convexity} nicely unifies the theory: it guarantees the existence and uniqueness of the three gradient flows and it also provides a bound on their rate of convergence to the posterior. In the $L^2$ setting  one can  show that positive geodesic convexity is equivalent to the posterior satisfying a Poincar\'e inequality, and also to the existence of a spectral gap ---see subsection \ref{ssec:geodesicdmu}. On the other hand, the geodesic convexity of $\jkl$ and $\jchi$ in Wasserstein space is determined by the Ricci curvature of the manifold, as well as by the likelihood function and prior density  ---see \eqref{functionalkl}, \eqref{fkldefinition}, \eqref{functionalchi}, \eqref{fchi}. Typically the three functionals $\jkl,$ $\jchi,$ and $D^\mu$  will have different levels of geodesic convexity, and establishing a sharp bound on each of them may not be equally tractable.

The theory of gradient flows and optimal transport gives, for each of the flows, an associated Fokker-Planck partial differential equation (PDE) that governs the evolution of densities \cite{japoneses2011displacement}, \cite{santambrogio2015optimal}. Such PDEs are typically costly to discretize if the parameter space is of moderate or high dimension, but they may be used in small-dimensional problems as a way to define tempering schemes. Here we do not explore this idea any further. Instead, we focus on the (nonlinear) diffusion processes associated with the PDEs. These diffusions are Langevin-type stochastic differential equations, whose evolving densities satisfy the Fokker-Planck equations. By construction, the invariant distribution of each of these diffusions is the sought posterior, and a bound on the rate of convergence of the diffusions to the posterior is given by the geodesic convexity of the corresponding functional. The gradient flow perspective automatically gives a sense in which the diffusions are optimal: the associated densities move locally (in Wasserstein or $L^2$ sense) in the direction of steepest descent of the functional. 
From this it immediately follows, for instance, that  the law of a standard Langevin diffusion in Euclidean space evolves locally in Wasserstein space in the direction that minimizes Kullback-Leibler, and that it also evolves locally in $L^2$ in the direction that minimizes the Dirichlet energy.

The MCMC methodology allows to use a proposal based on a discretization of a diffusion ---combined with an accept-reject mechanism to remove the discretization bias--- to produce, in the large-time asymptotic, correlated posterior samples. Heuristically, the rate of convergence of the un-discretized diffusion may guide the choice of proposal. Proposals based on Langevin diffusions  were first suggested in \cite{besag1994comments}, and the exponential ergodicity of the resulting algorithms was analysed in \cite{roberts1996exponential}. The paper \cite{girolami2011riemann} considered changing the metric on the parameter space in order to accelerate MCMC algorithms by taking into account the geometric structure that the posterior  defines in the parameter space. This led to a new family of Riemannian MCMC algorithms. Our paper is concerned with the study of un-discretized diffusions; the effect of the accept-reject mechanism on rates and ergodicity of MCMC methods will be studied elsewhere.  We suggest that a way to guide the choice of metric of Riemannian MCMC methods is to choose the one that leads to a faster rate of convergence of the diffusion under certain constraints. We emphasize that despite working with un-discretized diffusions, our guidance for choice of proposals accounts for the fact that discretization will eventually be needed. Our criterion weeds out choices of metric that lead to diffusions that achieve fast rate of convergence by merely speeding-up the drift. This is crucial, since a larger drift typically leads to a larger discretization error, and therefore to more rejections in the MCMC accept-reject mechanism in order to remove the bias in the discrete chain. This important constraint on the size of the drift seems to have been overlooked in existing continuous-time analyses of MCMC methods.

In summary, the following points highlight the key elements and common structure of the variational formulations of the Bayesian update and of the study of the associated gradient flows:
\begin{itemize}
	\item The posterior can be characterized as the minimizer of different functionals on probability measures or densities. 
	\item One can then study the gradient flows of these functionals with respect to a metric on the space of probability measures or densities; the resulting curve is a curve of maximal slope and its endpoint is the posterior.
	\item The gradient flows are characterized by a Fokker-Planck PDE that governs the evolution of the density of an associated diffusion process. 
	\item By studying the convexity of the functionals (with respect to a given metric) one can obtain rates of convergence of the gradient flows towards the posterior. In particular, the level of convexity determines the speed of convergence of the densities of the associated diffusion process towards the posterior, and hence can be used as a criterion to guide the choice of proposals for MCMC methods;  here we emphasize that care must be taken when comparing different diffusions if a higher speed of convergence is at the cost of a more expensive discretization. 
\end{itemize}

The ideas in this paper immediately extend beyond the Bayesian interpretation stressed here to any application (e.g. the study of conditioned diffusions) where a measure of interest is defined in terms of a reference measure and a change of measure. Also, we consider only Kullback-Leibler and $\chi^2$ prior penalizations to define the functionals $\jkl$ and $\jchi$, but it would be possible to extend the analysis to the family of $m$-divergences introduced in \cite{japoneses2011displacement}. Kullback-Leibler and $\chi^2$ prior penalization correspond to  $m\to1$ and $m=2$ within this family.  
In what follows we point out some of the features of the different functionals and gradient flows that we consider in this paper.

\nc
\subsection{Comparison of Functionals and Flows}\label{comparison} We now provide a comparison of the three choices of functionals that we consider. 

\begin{enumerate}
\item The two gradient flows in Wasserstein space (arising from the functionals $\jkl$ and $\jchi$) are fundamentally connected with the variational formulation: these variational formulations can be used to \emph{define}  posterior-type measures via a penalization of deviations from the prior and deviations from the data in situations where establishing the existence of conditional distributions by disintegration of measures is technically demanding. On the other hand, the variational formulation for the Dirichlet energy is less natural and requires previous knowledge of the posterior.
\item The precise level of geodesic convexity of the functionals $\jkl$ (and $\jchi$) can be computed from point evaluation of the Ricci tensor (of the parameter space) and derivatives of the densities. In particular, knowledge of the underlying metric suffices to compute these quantities. In contrast, establishing a sharp Poincar\'e inequality ---the level of geodesic convexity of the Dirichlet energy in $L^2(\M,\mu)$--- is in practice unfeasible, as it effectively requires solving an infinite dimensional optimization problem. It is for this reason ---and because of the explicit dependence of the convexity in Wasserstein space with the geometry induced by the manifold metric tensor--- that our analysis of the choice of metric in Riemannian MCMC methods is based on the $\jkl$ functional (see section \ref{sec:RMCMC}, and in particular Theorem \ref{theoremrmcmc}). 
\item On the flip side of point 2,  a Poincar\'e inequality for the posterior with a not necessarily optimal constant can be established using only tail information. In particular, even when the functional $\jkl$ is not geodesically convex in Wasserstein space, one may still be able to obtain a Poincar\'e inequality (see subsection \ref{sssec:glandau} for an example).
\item In contrast to the diffusions arising from the $\jkl$ or Dirichlet flows, the stochastic processes arising from the $\jchi$ formulation are inhomogeneous, and hence  simulation seems more challenging unless further structure is assumed on the prior measure and likelihood function. Also, the evolution of densities of the gradient flow of $\jchi$ in Wasserstein space is given by a porous medium PDE. 

\end{enumerate}


\subsection{Outline}
The rest of the paper is organized as follows. Section \ref{sec:preliminaries} contains some background material on the Wasserstein space, geodesic convexity of functionals, and gradient flows in metric spaces.  The core of the paper is section \ref{sec:section3mainbody}, where we study the geodesic convexity, PDEs, and diffusions associated with each of the three functionals $\jkl,$ $\jchi$, and $D^\mu.$
In section \ref{sec:RMCMC} we consider an application of the theory to the choice of metric in Riemannian MCMC methods \cite{girolami2011riemann}, and in section \ref{sec:examplesgeodesic} we illustrate the main concepts and ideas through examples arising in Bayesian formulations of semi-supervised learning \cite{trillos2017continuum}, \cite{trillos2017consistency}, \cite{bertozziuncertainty}. We close in section \ref{sec:conclusions} by summarizing our main contributions and pointing to open directions.

\subsection{Set-up and Notation} 
\label{sec:notation}
 $(\M,g)$ will denote a smooth connected $m$-dimensional Riemannian manifold with metric tensor $g$ representing the parameter space. We will denote by $d$ the associated Riemannian distance, and assume that $(\M,d)$ is a complete metric space. By the Hopf-Rinow theorem it follows that $\M$ is a \emph{ geodesic space} ---we refer to subsection \eqref{sec:geodesicspaces} for a discussion on geodesic spaces and their relevance here. We denote by $vol_g$ the associated volume form. To emphasize the dependence of differential operators on the metric with which $\M$ is endowed, we write $\nabla_g$, $\div_g,$ $\text{Hess}_g$ and $\Delta_g$ for the gradient, divergence, Hessian, and Laplace Beltrami operators on $(\M,g).$ The reader not versed in Riemannian geometry may focus on the case $\M=\R^m$ with the usual metric tensor,  in which case $d$ is the Euclidean distance and $dvol_g = dx$ is the Lebesgue measure. However, in section \ref{sec:RMCMC} where we discuss applications to Riemannian MCMC, we endow $\R^m$ with a general metric tensor $g$ and hence familiarity with some notions from differential geometry is desirable.
 
We denote by $\p(\M)$ the space of probability measures on $\M$ (endowed with the Borel $\sigma$-algebra). We will be concerned with the update of a \emph{prior} probability measure $\pi \in \p(\M)$ ---that represents various degrees of belief on the value of a quantity or parameter of interest--- into a \emph{posterior} probability measure $\mu \in \p(\M)$, based on observed data $y$.
We will assume that the prior is defined as a change of measure from $vol_g,$ and that the posterior is defined as a change of measure from $\pi$ as follows:
\begin{equation}\label{priorposteriordefinition}
\pi = e^{-\Psi} vol_g, \quad \quad \mu \propto e^{-\phi}\pi.
\end{equation}
The data is incorporated in the Bayesian update through the negative log-likelihood function $\phi(\cdot) = \phi(\cdot;y).$  

\section{Preliminaries}

In this section we provide some background material. The Wasserstein space, and the notion of $\lambda$-geodesic convexity of functionals are reviewed in subsection \ref{sec:geodesicspaces}. Gradient flows in metric spaces are reviewed in subsection \ref{sec:gradflows}.
\label{sec:preliminaries}
\subsection{Geodesic Spaces and Geodesic Convexity of Functionals}
\label{sec:geodesicspaces}
A geodesic space $(X,d_X)$ is a metric space with a notion of length of curves that is compatible with the metric, and where every two points in the space can be connected by a curve whose length achieves the distance between the points (see \cite{burago} for more details). Geodesic spaces constitute a large family of metric spaces with a rich theory of gradient flows.   Here we consider three geodesic spaces. First, the \emph{base space} $(\M,d)$, i.e. the manifold $\M$ equipped with its Riemannian distance. Second, the space $\ptwo(\M)$ of square integrable Borel probability measures defined on $\M$, endowed with the Wasserstein distance $\W_2$. Third, the space of functions $f\in L^2(\M,\mu)$, with $\int_\M f d\mu =1,$ equipped with the $L^2(\M,\mu)$ norm. 

We spell out the definitions of $\ptwo(\M)$ and $\W_2$:
\begin{align}\label{defwasserstein}
\begin{split}
\ptwo(\M)&:= \Bigl\{ \nu \in \p(\M) \: : \:   \int_{\M}d^2(x, x_0) d \nu(x) < \infty, \,\, \text{ for some } x_0 \in \M \Bigr\}, \\
\W_2^2(\nu_1,\nu_2) &:= \inf_\alpha \int_{\M \times \M} d(x,y)^2 d\alpha(x,y), \quad \nu_1, \nu_2 \in \ptwo(\M).
\end{split}
\end{align}
The infimum in the previous display is taken over all \textit{transportation plans} between $\nu_1$ and $\nu_2$, i.e. over $\alpha\in \p(\M\times \M)$ with marginals $\nu_1$ and $\nu_2$ on the first and second factors. The space $(\ptwo(\M),\W_2 )$ is indeed a geodesic space: geodesics in $(\ptwo(\M), \W_2)$ are induced by those in $(\M,d)$. All it takes to construct a geodesic connecting $\nu_0\in \ptwo(\M)$ and $\nu_1 \in \ptwo(\M)$ is to find an optimal transport plan between $\nu_0$ and $\nu_1$ to determine source locations and target locations, and then transport the mass along geodesics in $\M$ (see \cite{villani2003topics} and \cite{santambrogio2015optimal}). 

The space of functions $f\in L^2(\M,\mu)$, with $\int_\M f d\mu =1,$ equipped with the $L^2(\M,\mu)$ norm is also a geodesic space, where a constant speed geodesic connecting $f_0$ and $f_1$ is given by linear interpolation: $t \in [0,1] \mapsto (1-t)f_0 + t f_1$.

We will consider several functionals $E:X\to \R\cup\{\infty\}$ throughout the paper. They will all be defined in one of our three geodesic spaces ---that is, $X=\M$, $X=\ptwo(\M)$ or $X=L^2(\M,\mu)$. Important examples will be, respectively:
\begin{enumerate}
\item  Functions $\Psi:\M\to  \R\cup\{\infty\}.$
\item The Kullback-Leibler and $\chi^2$ divergences $\dkl(\cdot\| \pi),$  $\dchi(\cdot\|\pi):\p(\M) \to [0,\infty],$ where $\pi$ is a given (prior) measure and,  for $\nu_1, \nu_2 \in \p(\M),$ 
\begin{equation}\label{kldefinition}
 \dkl(\nu_1 \| \nu_2) := \begin{cases}  \int_{\M}  \frac{d \nu_1}{d \nu_2}(u) \log\left( \frac{d\nu_1 }{d \nu_2}(u) \right)d \nu_2(u),  &  \nu_1 \ll \nu_2,  \\ \infty, & \text{otherwise},   \end{cases}  
 \end{equation}
\begin{equation}\label{chidefinition}
\dchi(\nu_1 \| \nu_2) := \begin{cases}  \int_{\M} \Bigl( \frac{d \nu_1}{d \nu_2}(u) -1\Bigr)^2 d \nu_2(u),  &  \nu_1 \ll \nu_2,  \\ \infty, & \text{otherwise}; \end{cases} 
\end{equation}  
and the potential-type functional $J:\p(\M) \to \R\cup\{\infty\}$ 
given by 
$$J(\nu) := \int_\M h(u) d\nu(u),$$
where $h$ is a given potential function.
\item 
 The \textit{Dirichlet} energy $D^\mu : L^2(\M, \mu) \rightarrow [0, \infty]$ defined by
\begin{equation}\label{dirichletenergy}
 D^{\mu}(f) = \begin{cases}   \int_{\M} \| \nabla_g f(u) \|^2 d \mu(u), & f \in L^2(\M, \mu)\cap H^1(\M),   \\ +\infty, & \text{otherwise}.      \end{cases},
 \end{equation}
Recall that here and throughout, $\nabla_g$ denotes the gradient in $(\M,g)$ and $\lVert \cdot \rVert$ is the norm on each tangent space $\mathcal{T}_x \M$.
\end{enumerate}

A crucial unifying concept will be that of $\lambda$-geodesic convexity of functionals. We recall it here:
\begin{defn}
Let $(X,d_X)$ be a geodesic space and let $\lambda \in \R$. A functional $E:X \to \R \cup \{\infty\}$ is called $\lambda$-geodesically convex provided that  for any $x_0, x_1 \in X$ there exists a constant speed geodesic $t \in [0,1] \mapsto \gamma(t) \in X$ such that $\gamma(0) = x_0,$ $\gamma(1) = x_1,$ and 
\begin{equation*}
E\bigl(\gamma(t)\bigr) \le (1-t) E(x_0) + tE(x_1) - \lambda\frac{t(1-t)}{2} d_X^2(x_0,x_1), \quad \forall t\in [0,1].
\end{equation*}
\end{defn}
The following remark characterizes the $\lambda$-convexity of functionals when $X=\M$.
\begin{rem}\label{remarkconvexity}
Let $\Psi \in C^2(\M)$ so that we can define its Hessian at all points in $\M$ (see the proof of Theorem \ref{theoremrmcmc} in the appendix for the definition). 
Then the following conditions are equivalent:
\begin{enumerate}[(i)]
\item $\Psi$ is $\lambda$-geodesically convex.
\item $\text{\emph {Hess}}_g\, \Psi_x(v,v) \geq \lambda$ for all $x \in \M$ and all unit vectors $v\in T_x \M.$
\end{enumerate}
If $(\M,d)$ is the Euclidean space, (i) and (ii) are also equivalent to:
\begin{enumerate}[(iii)]
\item $\Psi- \frac{\lambda}{2}\lvert \cdot \rvert^2$ is a convex function.
\end{enumerate}
This latter condition is known in the optimization literature as strong convexity. 
\end{rem}

\subsection{Gradient Flows in Metric Spaces}
\label{sec:gradflows}
In this subsection we review the basic concepts needed to define gradient flows in a metric space $(X,d_X)$. We follow  Chapter 8 of \cite{santambrogio2015optimal}; a standard technical reference is \cite{ambrosio2008gradient}. 

To guide the reader, we first recall the formulation of gradient flows in Euclidean space, where $X = \R^d$ and $d_X$ is the Euclidean metric. Let $E : \R^d \rightarrow \R$ be a differentiable function, and consider the equation
\begin{align}\label{ODE}
\begin{cases}
\dot{x}(t) &=\,\, - \nabla E\bigl(x(t)\bigr), \quad t \geq 0,\\
x(0) &=\,\, x_0.
\end{cases}
\end{align}
Then, the solution  $x$   to \eqref{ODE} is the \textit{gradient flow} of $E$ in Euclidean space with initial condition $x_0$; it is a curve whose tangent vector at every point in time is the negative of the gradient of the function $E$ at that time. In order to generalize the notion of a gradient flow to functionals defined on more general metric spaces, and in particular when the metric space has no differential structure, we reformulate \eqref{ODE} in integral form by using that $\frac{d}{dt}E\bigl(x(t)\bigr) = \langle \nabla E\bigl(x(t)\bigr), \dot{x}(t)\rangle = - \frac{1}{2} |\dot{x}(t)|^2 - \frac{1}{2}|\nabla E\bigr(x(t)\bigl) |^2$ as follows: 
\begin{equation}
E(x_0) = E\bigl(x(t)\bigr) + \frac{1}{2}\int_{0}^{t} \left|\dot{x}(r)\right|^2 dr + \frac{1}{2} \int_{0}^t \left|  \nabla E \bigl(x(r)\bigr) \right|^2 dr, \quad t >0.
\label{EDE}
\end{equation}
This identity, known as energy dissipation equality,  is equivalent to \eqref{ODE} ---see Chapter 8 of \cite{santambrogio2015optimal} for further details and other possible formulations. Crucially \eqref{EDE} involves notions that can be defined in an arbitrary metric space $(X,d_X)$: the metric derivative of a curve $t \mapsto x(t) \in X$ is given by
\[ | \dot{x}(t)| := \lim_{s \rightarrow t}  \frac{d_X\bigl(x(t), x(s)\bigr)}{| s- t|},   \] 
and the {\em slope} of a functional  $E:X\to \R \cup \{\infty\}$ is defined as the map $|\nabla E|:\{x\in X: E(x)<\infty\} \to \R \cup \{\infty\}$  given by 
\begin{equation*}
|\nabla E| (x) := \limsup_{y\to x} \frac{\bigl(E(x) - E(y)\bigr)^+}{d_X(x,y)}.
\end{equation*}
The identity \eqref{EDE} is the standard way to introduce gradient flows in arbitrary metric spaces. In this paper we consider gradient flows in $L^2$ and Wasserstein spaces, where the notion of tangent vector is available. 
$L^2$ has Hilbert space structure, whereas the Wasserstein space can be seen as an infinite dimensional manifold (see \cite{ambrosio2008gradient}, \cite{santambrogio2015optimal}).

%

\section{Variational Characterizations of the Posterior and Gradient Flows}\label{sec:section3mainbody}

In this section we lay out the main elements of the theory of variational formulations and gradient flows in regards to the Bayesian update. Subsection \ref{ssec:variational} details three variational formulations defined in terms of the functionals $\jkl$, $\jchi$ and the Dirichlet energy $D^\mu$. Subsection \ref{ssec:geodesic} studies the geodesic convexity of $\jkl$ and $\jchi$  in Wasserstein space and of $D^\mu$ in $L^2$. Finally, subsection \ref{ssec:PDEdiffusion} collects the  PDEs that characterize the gradient flows, as well as the corresponding diffusion processes.

\subsection{Variational Formulation of the Bayesian Update}\label{ssec:variational}
The variational formulation of the posterior as the minimizer of $\jkl$ and $\jchi$ share the same structure and will be outlined first. The variational formulation in terms of the Dirichlet energy will be given below.

\subsubsection{The Functionals $\jkl$ and $\jchi$}\label{ssec:varjkl}
In mathematical analysis \cite{jordan199618} and probability theory \cite{dupuis2011weak} it is often useful to note that a probability measure $\mu$ defined by
\begin{equation}\label{eq:interpretation}
\mu(du) =\frac{1}{Z} \exp\bigl(-\phi(u)\bigr)\pi(du)
\end{equation}
is the minimizer of the functional
\begin{equation}\label{functionalkl}
\jkl (\nu)  := \dkl(\nu \| \pi ) + \fkl (\nu;\phi), \quad \nu \in \p(\M),
\end{equation}
where 
\begin{equation}\label{fkldefinition}
 \fkl(\nu;\phi)  :=  \int_{\M} \phi(u) d\nu(u),
\end{equation}
and the integral is interpreted as $+\infty$  if  $\phi$ is not integrable with respect to $\nu$. In physical terms, the Kullback-Leibler divergence represents an internal energy, $\fkl$ represents a potential energy, and the constant $Z$ is known as the partition function. Here we are concerned with a statistical interpretation of equation \eqref{eq:interpretation}, and view it as defining a posterior measure as a change of measure from a prior measure. In this context, the Kullback-Leibler term $\dkl(\cdot\|\pi)$ in \eqref{functionalkl} represents a penalization of deviations from prior beliefs, the term $ \fkl(\nu;\phi)$ penalizes deviations from the data, and the normalizing constant $Z$ represents the marginal likelihood. For brevity, we will henceforth suppress the data $y$ from the negative log-likelihood function $\phi$, writing $\phi(u)$ instead of $\phi(u;y).$  

We remark that the fact that $\mu$ minimizes $\jkl$ follows immediately from the identity 
\begin{equation}\label{dkljkl}
\dkl(\cdot\|\mu) =  \jkl(\cdot) + \log Z.
\end{equation}
Minimizing $\jkl(\cdot)$ or $\dkl(\cdot\|\mu)$ is thus equivalent, but the functional $\jkl$ makes apparent the roles of the prior and the likelihood.

The posterior $\mu$ also minimizes the functional 
\begin{equation}\label{functionalchi}
\jchi (\nu)  := \dchi(\nu \| \pi ) + \fchi (\nu;\phi), \quad \nu \in \p(\M),
\end{equation}
where 
\begin{equation}\label{fchi}
\fchi(\nu;\phi)  :=  \int_{\M} \tilde{\phi}(u) d\nu(u), \quad \tilde{\phi} = g\bigl(\exp(\phi(u))\bigr), \quad g(t) := t-1,\quad t>0.
\end{equation} 
We refer to \cite{japoneses2011displacement} for details. Note that both $\jkl$ and $\jchi$ are defined in terms of the two starting points of the Bayesian update: the prior $\pi$ and the negative log-likelihood $\phi.$ The associated variational formulations suggest a way to \emph{define}  posterior-type measures based on these two ingredients in scenarios where establishing the existence of conditional distributions via desintegration of measures is technically demanding. This appealing feature of the two variational formulations above is not shared by the one described in the next subsection. 
\subsubsection{The Dirichlet Energy $D^\mu$}  \label{ssec:vardirichlet}
Let now the posterior $\mu$ be given, and consider the space $L^2(\M, \mu)$ of functions defined on $\M$ which are square integrable with respect to $\mu$.  Recall the Dirichlet energy  $$D^{\mu}(f) =   \int_{\M} \| \nabla_g f(u) \|^2 d \mu(u),$$
introduced in equation \eqref{dirichletenergy}. Now, since the measure $\mu$ can be characterized as  the probability measure with density $\rho_\mu \equiv 1$ a.s. with respect to $\mu,$  it follows that the posterior density $\rho_\mu\equiv1$ is the minimizer of the Dirichlet energy $D^\mu$ over probability densities $\rho\in L^2(\M, \mu)$ with $\int_\M \rho d \mu = 1.$ \nc

\subsection{Geodesic Convexity and Functional Inequalities}\label{ssec:geodesic}
In this section we study the geodesic convexity of the functionals $\jkl$, $\jchi$, and $D^\mu$. The geodesic convexity of $\jkl$ and $\jchi$ in Wasserstein space is considered first, and will be followed by the geodesic convexity of $D^\mu$ in $L^2$. We will show the equivalence of the latter to the posterior satisfying a Poincar\'e inequality.

\subsubsection{Geodesic Convexity of $\jkl$ and $\jchi$}\label{ssec:geodesicjkljchi}
The next proposition can be found in   \cite{von2005transport} and \cite{sturm2006geometry}. It shows that the convexity of $\jkl$ can be determined by the so-called curvature-dimension condition ---a condition that involves the curvature of the manifold and the Hessian of the combined change of measure $\Psi + \phi.$ We recall the notation $\pi = e^{-\Psi} vol_g$ and $\mu \propto e^{-\phi} \pi$. 

\begin{prop}\label{propositionkl}
Suppose that $\Psi, \phi \in C^2(\M).$ Then $\jkl$ (or $\dkl(\cdot\| \mu)$) is $\lambda$-geodesically convex if, and only if, 
$$\text{\emph{Ric}}_g\,(v, v)  +  \text{\emph{ Hess}}_g\, \Psi (v,v) + \text{\emph{Hess}}_g\, \phi(v,v) \geq \lambda, \quad \forall x \in \M , \quad \forall v \in T_x \M \,\,\,\,  \text{with}  \,\,\,\, g(v,v) =1,$$
where $\text{Ric}_g$ denotes the Ricci curvature tensor.
\end{prop}

We recall that the Ricci curvature provides a way to quantify the disagreement between the geometry of a Riemannian manifold and that of ordinary Euclidean space. The Ricci tensor is defined as the trace of a map involving the Riemannian  curvature (see \cite{docarmo1992riemannian}).

The following example illustrates the geodesic convexity of $\dkl(\cdot\|\mu)$ for Gaussian $\mu.$

\begin{example}\label{ex:convexitygaussian}
Let $\mu = N(\theta, \Sigma)$ be a Gaussian measure in $\R^m$  (endowed with the Euclidean metric), with $\Sigma$ positive definite. Then $\dkl(\cdot \| \mu)$ is $1/\Lambda_{\max}(\Sigma)$-geodesically convex, where $\Lambda_{\max}(\Sigma)$ is the largest eigenvalue of $\Sigma.$ This follows immediately from the above, since here $\Psi (x) = \frac{1}{2}\langle x-\theta, \Sigma^{-1} (x-\theta)\rangle,$ and the Euclidean space is flat (its Ricci curvature is identically equal to zero). Note that the level of convexity of the functional depends only on the largest eigenvalue of the covariance, but not on the dimension $m$ of the underlying space.
\end{example}

The $\lambda$-convexity of  $\jkl$ guarantees the existence of the gradient flow of  $\jkl$ in Wasserstein space. Moreover, it determines the rate of convergence towards the posterior $\mu.$ Precisely, if $\mu_0$ is absolutely continuous with respect to $\mu,$ and if $\lambda>0$, then the gradient flow $t \in [0,\infty) \mapsto \mu_t$ of $\jkl$ with respect to the Wasserstein metric starting at $\mu_0$ is well defined and we have: 
\begin{align}\label{inequalities}
\begin{split}
\dkl(  \mu_t  \|  \mu ) &\leq e^{- \lambda t } \dkl(\mu_0\| \mu) , \quad t \geq 0, \\
 \W_2(\mu_t, \mu)^2  &\leq \lambda e^{- \lambda t }\dkl(\mu_0\| \mu) , \quad  t \geq 0.  
\end{split}
\end{align}
The second inequality, known as Talagrand inequality \cite{villani2003topics}, establishes a comparison between Wasserstein geometry and information geometry. It can be established directly combining the $\lambda$-geodesic convexity of $\jkl$ (for positive $\lambda$) with the first inequality.  From \eqref{inequalities} we see that a higher level of convexity of $\jkl$ allows to guarantee a faster rate of convergence towards the posterior distribution $\mu$.

We now turn to the geodesic convexity properties of $\jchi.$ We recall that $m$ denotes the dimension of the manifold $\M.$ The following proposition can be found in \cite[Theorem 4.1]{japoneses2011displacement}.

\begin{prop}
$\jchi$ is $\lambda$-geodesically convex if and only if both of the following two properties are satisfied:
\begin{enumerate}
\item $\text{\emph{Ric}}_g\,(v, v)  +  \text{\emph{Hess}}_g\, \Psi (v,v)  + \frac{1}{m+1} \langle \nabla_g \Psi , v  \rangle^2 \geq 0 , \quad \forall x \in \M , \quad \forall v \in T_x \M.  $
\item $\phi$ is $\lambda$-geodesically convex as a real valued function defined on $\M$.
\end{enumerate}
\end{prop}

There are two main conclusions we can extract from the previous proposition. First, that condition 1) is only related to the prior distribution $\pi$ whereas condition 2) is only related to the likelihood; in particular,  the convexity properties of $\jchi$ can indeed be studied by studying separately the prior and the likelihood (notice that the proposition gives an equivalence).  Secondly, notice that condition 1) is a qualitative property and if it is not met there is no hope that the functional $\jchi$ has any level of global convexity even when the likelihood function is a highly convex function. In addition, if 1) is satisfied, the convexity of $\phi$ determines completely the level of convexity of $\jchi$. These features are markedly different from the ones observed in the Kullback-Leibler case. 

As for the functional $\jkl$, one can establish the following functional inequalities, under the assumption of $\lambda$-geodesic convexity of $\jchi$  for $\lambda>0$: 
\begin{align}\label{inequalitiesjchi}
\begin{split}
\jchi(\mu_t) - \jchi(\mu) &\leq e^{- \lambda t } \bigl( \jchi(\mu_0) - \jchi(\mu)\bigr) , \quad t \geq 0, \\
 \W_2(\mu_t, \mu)^2  &\leq \lambda e^{- \lambda t }\bigl(\jchi(\mu_0) - \jchi(\mu)\bigr) , \quad  t \geq 0.  
\end{split}
\end{align}

The above inequalities exhibit the fact that a higher level of convexity of $\jchi$ guarantees a faster convergence towards the posterior distribution $\mu$.

\subsubsection{Geodesic Convexity of Dirichlet Energy} \label{ssec:geodesicdmu}
We now study the geodesic convexity of the Dirichlet energy functional  defined in equation \eqref{dirichletenergy}.  In what follows we denote by $\| \cdot\|$ the $L^2$ norm with respect to $\mu.$ 
Let us start recalling Poincar\'{e} inequality. 

\begin{defn}
We say that a Borel probability measure $\mu$ on $\M$ has a  Poincar\'{e} inequality with constant $\lambda$ if for every $f \in L^2(\M, \mu)$ satisfying $\int_{\M}f d \mu=0$ we have
\[ \lVert f \rVert_{\mu}^2 \leq \frac{1}{\lambda} D^\mu(f).    \]
\end{defn}

We now show that Poincar\'{e} inequalities are directly related to the geodesic convexity of the functional $D_{\mu}$ in the $L^2(\M,\mu)$ space. \begin{prop}
Let $\lambda$ be a positive real number and let $\mu$ be a Borel probability measure on $\M$. Then, the measure $\mu$ has a Poincar\'{e} inequality with constant $\lambda$ if and only if the functional $D^\mu$ is $2\lambda$-geodesically convex in the space of functions $f \in L^2(\M, \mu)$ satisfying $\int f d \mu =1 $.
\end{prop}

\begin{proof}
First of all we claim that 
\begin{equation}\label{Claim1}
D^\mu(t f_0+ (1-t) f_1) + t(1-t)D^\mu (f_0 - f_1)   = t D^\mu(f_0) + (1-t) D^\mu(f_1), 
\end{equation}
for all $f_0, f_1 \in L^2(\M,\mu)$ and every $t \in [0,1]$. To see this, it is enough to assume that both $D^\mu(f_0)$ and $D^\mu(f_1)$ are finite and then notice that equality \eqref{Claim1} follows from the easily verifiable fact that for an arbitrary Hilbert space $V$ with induced norm $|\cdot|$ one has
\[  |t v_0 + (1-t)v_1|^2 + t(1-t)|v_0-v_1|^2= t|v_0|^2 +(1-t)|v_1|^2  , \quad \forall v_0, v_1 \in V, \quad \forall t \in [0,1]. \]


Now, suppose that $\mu$ has a Poincar\'{e} inequality with constant $\lambda$ and consider  two functions $f_0, f_1 \in L^2(\M, \mu)$ satisfying $\int_\M f_0 d\mu = \int_\M f_1 d\mu =1 $. Then, \eqref{Claim1} combined with Poincar\'{e} inequality (taking $f:= f_0 - f_1$) gives:
\begin{equation}
  D^\mu\bigl(t f_0 + (1-t)f_1 \bigr) + \lambda t(1-t)\lVert f_0 - f_1\rVert_\mu^2 \leq t D^\mu(f_0) + (1-t)D^\mu(f_1) ,
 \label{Claim2} 
\end{equation}
which is precisely the $2 \lambda$-geodesic convexity condition for $D^\mu$. 

Conversely, suppose that $D^\mu$ is $2 \lambda$-geodesic convex in the space of $L^2(\M,\mu)$ functions that integrate to one. Let $f \in L^2(\M, \mu)$ be such that $\int_{\M}f d \mu =0$ and without the loss of generality assume that $D^\mu(f)< \infty$ and that $\|f\|_\mu \not =0$. Under these conditions, the positive and negative parts of $f$, $f^+$ and $f^-$, satisfy $D^\mu(f^+), D^\mu(f^-) < \infty$ and $\int_{\M} f^+ d\mu = r = \int_{\M} f^- d \mu$ where $r>0$. The inequality 
\[ \|f\|_\mu^2 \leq \frac{1}{\lambda}D^\mu(f) \]
is obtained directly from \eqref{Claim1} and \eqref{Claim2} applied to 
 \[f_0 :=\frac{1}{r}f^- , \quad f_1:= \frac{1}{r} f^+,\quad t =1/2.\]
 \end{proof}

\begin{remark}
\label{Eigangap}
It is well known that the best Poincar\'e constant for a measure $\mu$ is equal to the smallest non-trivial eigenvalue of the operator $-\Delta_g^\mu$ defined formally as 

\begin{equation*}
- \Delta_g^\mu f:=-\frac{1}{Z} \, \text{\emph{div}}_g  \,\,  ( e^{-\phi- \psi} \nabla_g f )  ,
\end{equation*}
where $\text{\emph{\div}}_g$ and $\nabla_g$ are the divergence and gradient operators in $(\M,g)$. This eigenvalue can be written variationally as
\[ \lambda_{2}:=  \min_{f \in L^2(\M,\mu)} \frac{D_ {\mu}(f) }{\| f- f_\mu\|^2_\mu},\]
where
\[f_{\mu} := \int_\M f d\mu.\]	
\end{remark}

\begin{remark}
Spectral gaps are used in the theory of MCMC as a means to bound the asymptotic variance of empirical expectations \cite{kipnis1986central}.	
\end{remark}

Let us now consider $t \in (0,\infty) \mapsto \mu_t$ the flow of $D^\mu$ in $L^2(\M, \mu)$ with some initial condition $\frac{d\mu_0}{d\mu}= \rho_0$.  It is well known that this flow  coincides with that of the functional $\jkl$ in Wasserstein space. However, taking the Dirichlet-$L^2$ point of view, one can use a Poincar\'{e} inequality (i.e. the geodesic convexity of $D^\mu$) to deduce  the exponential convergence of $\mu_t$ towards $\mu$ in the $\chi^2$-sense. Indeed, let
\[\rho_t:= \frac{d\mu_t}{ d\mu}, \quad t \in(0,\infty). \]
A standard computation then shows that
\[\frac12 \frac{d}{dt} \lVert  \rho_t - 1 \rVert_\mu^2  = \int_{\M}(\rho_t - 1)\frac{\partial \rho }{\partial t}  d\mu = - \int_{\M} \lvert \nabla_g (\rho_t(u) - 1) \rvert^2 d \mu(u) \leq- \lambda  \lVert \rho_t -1 \rVert_\mu^2 . \]
In the second equality we have used that $\frac{\partial \rho}{\partial t} = \Delta^\mu_g \rho,$ as discussed in subsection \ref{ssec:PDEdiffusion} below. Hence by Gronwall's inequality, see e.g. \cite{teschl2012ordinary},
\[ \lVert  \rho_t - 1 \rVert_\mu \leq \exp(- \lambda t)\lVert \rho_0-1 \rVert_\mu, \quad t >0.\]

\subsection{PDEs and Diffusions}\label{ssec:PDEdiffusion}
Here we describe the PDEs that govern the evolution of densities of the three gradient flows, and the stochastic processes associated with these PDEs. We consider first the flows defined with the functionals $\jkl$ and $D^\mu$ and then the flow defined by the functional $\jchi.$ 

\subsubsection{$\jkl$-Wasserstein and $D^\mu$-$L^2(\M, \mu)$} \label{ssec:pdenormal}
It was shown in \cite{jordan1998variational} ---in the Euclidean setting and in the unweighted case $\pi = d x$--- that 
the gradient flow of the Kullback-Leibler functional $\dkl(\cdot\| \pi)$ in Wasserestein space produces a solution to the Fokker-Planck equation. More generally, under the convexity conditions guaranteeing the existence of the gradient flow $t \in (0,\infty) \mapsto \mu_t$ of $\dkl(\cdot \|\mu)$ (equivalently of $\jkl$) starting from $\mu_0\in \p(\M)$, the densities 
\[ \rho_t := \frac{d\mu_t}{d\mu}, \quad \theta_t := \frac{d\mu_t}{d vol_g}, \quad t\in (0,\infty)  \]
satisfy (formally) the following Fokker-Planck equations 
\begin{equation}\label{weightedlaplacian}
\frac{\partial\rho}{\partial t} = \Delta^\mu_g \rho.
\end{equation}
\begin{equation}
\frac{\partial \theta}{\partial t} = \Delta_g\theta +  \div_g\bigl( \theta(\nabla_g \phi + \nabla_g \Psi) \bigr).
\label{eqndensities}
\end{equation}
Equation \eqref{eqndensities} can be identified as the evolution of the densities (w.r.t. $d vol_g$) of the diffusion
\begin{equation}
dX_t = -\nabla_g\Bigl(\Psi(X_t) + \phi(X_t) \Bigr)\, dt + \sqrt{2} dB_t^g,
\label{diffusion}
\end{equation}
where $B^g$ denotes a Brownian motion defined on $(\M,g)$ and $\nabla_g$ is the gradient on $(\M, g)$. Naturally, the $D^\mu$ flow in $L^2$  has the same associated Fokker-Planck equation \eqref{weightedlaplacian} and diffusion process \eqref{diffusion}. 


%

\subsubsection{$\jchi$-Wasserstein}\label{ssec:PDEporous}
The PDE satisfied (formally) by the densities 
\[ \tilde{\rho}_t := \frac{d\mu_t}{ d \pi}, \]
of the $\jchi$-Wasserstein flow $t \in(0,\infty) \mapsto \mu_t$ is the (weighted) porous medium equation:
\begin{align}
&\frac{\partial \tilde{\rho}}{\partial t} =  \Delta^{\pi}_g \tilde{\rho}^2 + \div^\pi_g(\tilde{\rho} \nabla_g \phi),
\label{WPME}
\end{align}
where the weighted Laplacian and divergence are defined formally as
\begin{align}
\begin{split}
\Delta_g^\pi f &:= \Delta_g f - \langle \nabla_g f, \nabla_g \Psi \rangle, \quad  \\
\div^\pi_g F&:= \div_g F - \langle F, \nabla_g \Psi\rangle.
\end{split}
\label{IdentitiesWeightedDiv}
\end{align}

Consider now the stochastic process $\{ u_t \}_{t \geq 0}$ formally defined as the solution to the nonlinear diffusion
\begin{align}
\begin{split}
& dX_t =  -\bigl( \tilde \rho(t,X_t) \nabla_g \Psi( X_t) +  \nabla_g \phi(X_t ) \bigr)dt  + \sqrt{2\tilde{\rho}(t, X_t)}  d B_t^g, \quad  u_0 \sim \rho_0,
\end{split}
\label{WPMEDiff}
\end{align}
where $\tilde{\rho}$ is the solution to \eqref{WPME}. Let $\theta_t$ be the evolution of the densities (with respect to $d vol_g$ ) of the above diffusion. Then a formal computation shows that $\theta$
satisfies the Fokker-Planck equation:
\[  \frac{\partial \theta}{\partial t} = -  \div_g \Bigl( \theta \bigl( -\tilde{\rho} \nabla_g \Psi - \nabla_g \phi \bigr)   \Bigr) +  \Delta_g(  \tilde{\rho} \theta ) . \]
If we let $\beta= \frac{1}{Z}\exp(- \Psi)\theta $ we see, using \eqref{IdentitiesWeightedDiv}, that
\begin{align*}
\Delta^\pi_g \beta^2 &= e^{\Psi} \Bigl( \Delta_g(\beta^2 e^{-\Psi}) + \div_g \bigl((\beta^2 e^{-\Psi}) \nabla_g \Psi\bigr) \Bigr), \\
\div^\pi_g(\beta \nabla_g \phi) &= e^{\Psi} \div_g \bigl((\beta e^{-\Psi}) \nabla_g \Psi \bigr),
\end{align*}
implying that the distributions of the stochastic process \eqref{WPMEDiff} are those generated by the gradient flow of $\jchi$ in Wasserstein space. 

\begin{remark}
In contrast with the Langevin diffusion \eqref{diffusion}, the process \eqref{WPMEDiff} is defined in terms of the solution of the equation satisfied by its densities. In particular, if one wanted to simulate \eqref{WPMEDiff} one would need to know the solution of \eqref{WPME} before hand. 
\end{remark}
\nc

\section{Application: Sampling and Riemannian MCMC}
\label{sec:RMCMC}
So far we have treated the Riemannian manifold $(\M,g) $ as fixed. In this section we take a different perspective and treat the metric $g$ as a free parameter. Precisely, we will now consider a family of gradient flows of the functional $\jkl$ with respect to Wasserstein distances induced by different metrics $g$ on the parameter space. We do this motivated by the so called Riemannian MCMC methods for sampling, where a change of metric in the base space is introduced in order to produce Langevin-type proposals that are adapted to the geometric features of the target, thereby exploring regions of interest and accelerating the convergence of the chain to the posterior. There are different heuristics regarding the choice of metric  (see \cite{girolami2011riemann}), but no principled way to compare different metrics and rank their performance for sampling purposes. With the developments presented in this paper we propose one such principled criterion as we describe below. We restrict our attention to the case $\M = \R^m$. 

Let $g$ be a Riemannian metric tensor on $\R^m$ defined via
\[ g_x(u,v):= \langle G(x) u  , v \rangle, \quad u,v \in \mathcal{T}_x \M ,  \]
where for every $x \in \R^m$, $G(x)$ is a $m \times m$ positive definite matrix. In what follows we identify $g$ with $G$ and refer to both as `the metric'  and we use terms such as $g$-geodesic, $g$-Wassertein distance, etc. to emphasize that the notions considered are being constructed using the metric $g$.  Let $d_g$ be the distance induced by the metric tensor $g$ and let $vol_g$ be the associated volume form. Notice that in terms of the Lebesgue measure and the metric $G$, we can write 
\[ d vol_g(x) = \sqrt{\det\bigl(G(x)\bigr) } dx.\]
We use the canonical basis for $\R^m$ as global chart for $\R^m$ and consider the canonical vector fields  $\frac{\partial}{\partial x_1}, \dots, \frac{\partial}{\partial x_m}$. The \textit{Christoffel symbols} associated to the Levi-Civita connection of the Riemannian manifold $(\R^m,g)$ can be written in terms of derivatives of the metric as
\begin{equation}\label{Christoffelsymbols}
\Gamma_{ij}^l = \frac{1}{2} \left(  \pX{j} G_{ki} + \pX{i}G_{kj} - \pX{k}G_{ij}  \right)G^{-1}_{lk}, 
\end{equation}
where in the right hand-side ---and in what follows---  we use Einstein's summation convention. The proof of the following result is in the Appendix. 

\begin{thm}\label{theoremrmcmc}
Let $F\in C^2(\R^m)$ and 
$$\mu(du) \propto  \exp\bigl(-F(u)\bigr) du.$$
The sharp constant $\lambda$ for which $\jkl$ (or $\dkl(\cdot\| \mu)$) is $\lambda$-geodesically convex in the $g$-Wasserstein distance is equal to
\[ \lambda_G := \inf_{x \in \R^m} \Lambda_{\min} \Bigl( G^{-1/2} \bigl(B + \text{\emph{Hess}}\, F -C \bigr)G^{-1/2}  \Bigr), \]
where  $\text{\emph {Hess}}\, F$ is the usual (Euclidean) Hessian matrix of $F,$ $B$ is the matrix with coordinates
\begin{equation}\label{def:B}
B_{ij}:=  \frac{\partial  \Gamma_{ij}^l }{\partial x_l} - \Gamma_{il}^k \Gamma_{jk}^l,
\end{equation}
and $C$ is the matrix with coordinates
\begin{equation}\label{def:C}
 C_{ij}:= \Gamma_{ij}^l \frac{\partial F}{\partial x_l}.
\end{equation}
Moreover, for any $a>0$, 
\begin{equation}\label{homogeneity}
\lambda_{aG} = \frac{1}{a} \lambda_G.
\end{equation}
\end{thm}

Note that $\lambda_G$ is a key quantity in evaluating the quality of a metric $G$ in building geometry-informed Langevin diffusions for sampling purposes, as it gives the exponential rate at which the evolution of probabilities built using the metric $G$ converges towards the posterior:  larger $\lambda_G$ corresponds to faster convergence. However, in order to establish a fair performance comparison, the metrics need to be scaled appropriately. Indeed a faster rate can be obtained by scaling  down  the metric (which can be thought of as time-rescaling), as it is clearly seen by the scaling property \eqref{homogeneity} of the functional $\lambda_G.$ It is important to note that scaling down the metric leads to a faster diffusion, but also makes its discretization more expensive.
Indeed the error of Euler discretizations is largely influenced by the Lipschitz constant of the drift. This motivates that a fair criterion for choosing the metric could be to maximize $\lambda_G$ with the constraint 
\begin{equation}\label{Criterion}
\text{Lip}( \nabla_g F) = \text{Lip}(G^{-1} \nabla F) \le 1,  
\end{equation}
since $\nabla_g F  = G^{-1}\nabla F$ (where $\nabla$ denotes the standard Euclidean gradient) is the drift of the diffusion \eqref{diffusion}. Note that the constraint  \eqref{Criterion} ensures that the metric cannot be scaled down arbitrarily while also guaranteeing that the discretizations do not become increasingly expensive. We remark that other constraints involving higher regularity requirements may be useful if higher order discretizations are desired.

\begin{remark}
The functional $\lambda_G$ can be used to determine the optimal metric among a certain subclass of metrics of interest satisfying the condition \eqref{Criterion}. For instance, it may be of interest to find the optimal constant metric $G$ (see Proposition \ref{example} below), or to find the best metric within a finite family of metrics. On the other hand the constraint \eqref{Criterion} forces feasible metrics to  induce diffusions that are not expensive to discretize. 
\end{remark}

%
%

To illustrate the previous remark we show that for a Gaussian target measure
the optimal preconditioner is, unsurprisingly, given by the Fisher information.  More precisely we have the following proposition:
\begin{prop}\label{example}
Let $\mu = N(0,\Sigma).$ Then 
\begin{equation}\label{optimalggauss}
G^* :=  \Sigma^{-1}
\end{equation}
maximizes $\lambda_G$ over the class of constant metrics $G$ satisfying $\| G^{-1} \Sigma^{-1} \|  \le 1,$ as in \eqref{Criterion}.
Moreover, the maximum value is 
$$\lambda_{G^*} = 1.$$

\begin{proof}
Suppose for the sake of contradiction that there exists a constant metric $G$ that satisfies condition \eqref{Criterion}, which in this case reads $\lVert G^{-1} \Sigma^{-1} \rVert \leq 1$ and is such that $  \lambda_G > \lambda_{G^*}.$

Let $u$ be a unit norm eigenvector of $G$ with eigenvalue $\lambda>0$. Notice that by definition of $\lambda_{G}$ we must have
\begin{equation}
\langle G^{-1/2} \Sigma^{-1} G^{-1/2} u, u \rangle \geq \lambda_{G} > \lambda_{G^*} =1.  
\label{auxDiscretOptimal}
\end{equation}

The left hand side of the above display can be rewritten as
\[ \langle G^{-1/2} \Sigma^{-1} G^{-1/2} u, u \rangle =   \langle  G^{-1}\Sigma^{-1}  G^{-1/2} u, G^{1/2} u \rangle    \]
and by Cauchy-Schwartz inequality we see that
\[ \langle  G^{-1}\Sigma^{-1}  G^{-1/2} u, G^{1/2} u \rangle \leq \lVert G^{-1}\Sigma^{-1}  G^{-1/2} u  \rVert \cdot \lVert G^{1/2} u\rVert \leq  \lVert  G^{-1/2} u  \rVert \cdot \lVert G^{1/2} u\rVert \]

Since $u $ is an eigenvector of $G$ with eigenvalue $\lambda$, it follows that $u$ is also an eigenvector of $G^{1/2}$ with eigenvalue $\sqrt{\lambda}$ and of $G^{-1/2}$ with eigenvalue $\frac{1}{\sqrt{\lambda}}$. Therefore the right hand side of the above display is equal to one. This however contradicts \eqref{auxDiscretOptimal}. From this we deduce the optimality of $G^*$ among feasible metrics. 

\end{proof}
\end{prop}

\begin{example}
	\label{Example2}
Suppose that $F(u) =\frac12 \langle \Sigma^{-1} u, u \rangle,$ with 
\begin{equation*}
\Sigma = 
\begin{bmatrix} 
1 & 0 \\
0 & \epsilon 
\end{bmatrix}.
\end{equation*}
Consider the optimal metric 
$$G^* = 
\begin{bmatrix} 
1 & 0 \\
0 & \epsilon^{-1} 
\end{bmatrix}$$
given by the previous proposition and the rescaled Euclidean metric $G_e$
\begin{equation*}
G_e = 
\begin{bmatrix} 
\epsilon^{-1} & 0 \\
0 & \epsilon^{-1} 
\end{bmatrix},
\end{equation*}
where the scalings have been chosen so that
$$\emph{Lip}(G^{*-1} \nabla F) = \emph{Lip} (G_e^{-1} \nabla F) = 1.$$
A calculation then shows that $\lambda_{G^*} = 1$ while $\lambda_{G_e} = \epsilon.$
Note that if the Euclidean metric is not rescaled by $\epsilon^{-1}$ ---violating the constraint \eqref{Criterion}---  then the same unit rate of convergence as with the metric $G^*$ is achieved. However, the drift of the associated diffusion 
$$dX_t = - \Sigma^{-1} X_t\, dt + \sqrt{2} \, dB_t $$
is of order $\epsilon^{-1},$  making the discretization increasingly expensive in the small $\epsilon$ limit. On the other hand, since both $G^{*-1}$ and $G_e^{-1}$ are of order $\epsilon$, the drifts for both associated diffusions are order $1$. This motivates our choice of constraint in equation \eqref{Criterion}. 
\end{example}

\section{Example: Semi-Supervised Learning}\label{sec:examplesgeodesic}

In this section we study the geodesic convexity of functionals arising in the Bayesian formulation of semi-supervised classification. Our purpose is to illustrate the concepts in a tangible setting, and to show that establishing sharp levels of geodesic convexity may be more tractable for some functionals than others. 

In semi-supervised classification one is interested in the following task: given a data cloud $X= \{ x_1, \dots, x_n \}$ together with (noisy) labels $y_i \in \{-1,1\}$ for some of the data points $x_i,$ $i \in \mathcal{Z} \subset \{1, \dots, n \}$, classify the unlabeled data points by assigning labels to them. Here we assume to have access to a weight matrix $W$ quantifying the level of similarity between the points in $X$. Thus, we focus on the graph-based approach to semi-supervised classification, which boils down to propagating the known labels to the whole cloud, using the geometry of the weighted graph $(X,W)$. We will investigate the existence and convergence of gradient flows for several Bayesian graph-based classification models proposed in \cite{bertozziuncertainty}. In the Bayesian approach, the geometric structure that the weighted graph imposes on the data cloud is used to build a prior on a latent space, and the noisy given labels are used to build the likelihood. The Bayesian solution to the classification problem is a measure on the latent space, that is then push-forwarded into a measure on the label space $\{-1,1\}^n$. This latter measure contains information on the most likely labels, and also provides a principled way to quantify the remaining uncertainty on the classification process.

Let $(X, W)$ then be a weighted graph, where $X= \{ x_1, \dots, x_n \}$ is the set of nodes of the graph and $W$ is the weight matrix between the points in $X$. All the entries of $W$ are non-negative real numbers and we assume that $W$ is symmetric. Let $L$ be the graph Laplacian matrix defined by
\[  L:= D - W, \]
where $D$ is the degree matrix of the weighted graph, i.e.,  the diagonal matrix with diagonal entries $D_{ii}= \sum_{j=1}^n W_{ij}$. The above corresponds to the \textit{unnormalized} graph Laplacian, but different normalizations are possible \cite{von2007tutorial}. The graph-Laplacian will be used in all the models below to favor prior draws of the latent variables that are consistent with the geometry of the data cloud.

\begin{rem}
	\label{convEigen}
	A special case of a weighted graph $(X,W)$ frequently found in the literature is that in which the points in $X$ are i.i.d. points sampled from some distribution on a manifold $\M$ embedded in $\R^d$, and the similarity matrix $W$ is obtained as
	\[  W_{ij} =  K \left(\frac{\lvert x_i - x_j \rvert}{ r }\right).\]
	In the above, $K$ is a compactly supported kernel function, $\lvert x_i - x_j \rvert$ is the Euclidean distance between the points $x_i$ and $x_j,$ and $ r >0$ is a parameter controlling data density. It can be shown (see \cite{burago2013graph} and \cite{TrillosBreakThrough}) that the smallest non-trivial eigenvalue of a rescaled version of the resulting graph Laplacian is close to the smallest non-trivial eigenvalue of a weighted Laplacian on the manifold, provided that $r$ is scaled with $n$ appropriately. 
\end{rem}

We will now study the probit and logistic models in subsection \ref{sssec:probit}, and then the Ginzburg-Landau model in \ref{sssec:glandau}.

\subsection{Probit and Logistic Models}\label{sssec:probit}
Traditionally, the \textit{probit} approach to semi-supervised learning is to classify the unlabeled data points by first optimizing the functional $G:\R^n\to \R$ given by
\begin{equation}\label{definitionGH}
G(u):= \frac{1}{2}\langle    L^{  \alpha }  u , u \rangle - \sum_{j \in \mathcal{Z}} \log\bigl(H( y_j u_j)  ; \gamma \bigr), \quad H(w;\gamma):= \int_{-\infty}^w \exp(-t^2/2\gamma^2)\, dt 
\end{equation}
over all $u \in \R^n$ satisfying $\sum_{i=1}^n u_i =0,$ and then thresholding the optimizer with the sign function;   the parameter $\alpha>0$ is used to  regularize the functions $u$.  The minimizer of the functional $G$ can be interpreted as the MAP (maximum a posteriori estimator) in the Bayesian formulation of probit semi-supervised learning (see \cite{bertozziuncertainty}) that we now recall: 

\textbf{Prior:}
Consider the subspace $U:= \{ u \in \R^n \: : \: \sum_{i=1}^n u_i =0  \}$ and let $\pi$ be the Gaussian measure on $U$ defined by
\begin{equation}\label{priorprobit}
\frac{d \pi }{du}(u) \propto \exp \left( - \frac{1}{2}\langle L^\alpha  u , u \rangle \right)=:\exp\bigl(-\Psi(u)\bigr).
\end{equation}
The measure $\pi$ is interpreted as a prior distribution on the space of real valued functions on the point cloud $X$ with average zero. Larger values of $\alpha>0$ force more regularization of the functions $u$.

\textbf{Likelihood function:}
For a fixed $u \in U$ and for $j \in \mathcal{Z}$ define
\[ y_j = S( u_j + \eta_j  ),  \]
where the $\eta_j$ are i.i.d. $N(0, \gamma^2),$ and $S$ is the sign function. This specifies the distribution of observed labels given the underlying latent variable $u$. We then define, for given data $y$, the negative log-density function
\begin{equation}\label{sssec:likelihoodprobit}
\phi(u; y) := - \sum_{j \in \mathcal{Z}}\log\bigl( H(y_j  u_j ; \gamma) \bigr) , \quad u \in U,
\end{equation}
where $H$ is given by \eqref{definitionGH}.

\textbf{Posterior distribution:}
As shown in \cite{bertozziuncertainty},  a  simple application of Bayes' rule  gives  the posterior distribution of $u$ given $y$ (denoted by $\mu^y$):
$$\frac{d\mu^y}{d\pi}(v) = \exp\bigl( -\phi(v;y) \bigr); \quad \frac{d\mu^y}{dv}(v) \propto \exp\bigl( -\Psi(v) - \phi(v;y)\bigr), $$
where $\Psi$ is given by \eqref{priorprobit}, and $\phi$ is given by \eqref{sssec:likelihoodprobit}.

\bigskip

From what has been discussed in the previous sections, the posterior $\mu^y$ can be characterized as the unique minimizer of the energy 
\begin{equation}\label{energyminimizationforlogit}
\jkl( \nu  ):= \dkl( \nu  \| \pi) + \int_{\R^n}\phi( u;y) d \nu(u), \quad  \nu \in \mathcal{P}(\R^n).
\end{equation}
Let us first consider the gradient flow of $\jkl$ with respect to the usual Wassertsein space (i.e. the one induced by the Euclidean distance).

We can study the geodesic convexity of this functional by studying independently the convexity properties of $ \dkl( \nu  \| \pi)$ and of $\phi( \cdot; y)$. Precisely:
\begin{enumerate}[i)]
	\item Since $\pi$ is a Gaussian measure with covariance $ L^{-\alpha} $,  Example \ref{ex:convexitygaussian} shows that $\dkl( \nu  \| \pi)$ is $(\Lambda_{\min}(L))^\alpha$-geodesically convex in Wasserstein space, where $\Lambda_{\min}(L)$ is the smallest non-trivial eigenvalue of $L.$
	\item The function $\phi(\cdot; y)$ is convex ---see the appendix of \cite{bertozziuncertainty}. Hence, the functional $\fkl(\nu) = \int_{\R^n}\phi( u;y) d \nu(u)$ is $0$-geodesicaly convex in Wasserstein space. 
\end{enumerate}
It then follows from Proposition \ref{propositionkl} that $\jkl$ is $(\Lambda_{\min}(L))^\alpha$-geodesically convex in Wasserstein space. As a consequence, if we consider $t \in [0,\infty) \mapsto \mu_t$, the gradient flow of $\jkl$ with respect to the Wasserstein distance starting at $\mu_0$ (an absolutely continuous measure with respect to $\mu$), geometric inequalities can be immediately obtained from \eqref{inequalities}; such inequalities will not deteriorate with $n$ ---see Remark \ref{convEigen}.

However, the diffusion associated to this flow is given by 
\begin{equation}
dX_t = (-L^{ \alpha }X_t - \nabla \phi(X;y))\, dt  + \sqrt{2} dB_t,  
\label{difubasico}
\end{equation}
and in particular its drift (more precisely the term $L^\alpha X_t$) deteriorates as $n$ gets larger. Notice that if we wanted to control the cost of discretization by rescaling the Euclidean metric (as exhibited in Example \ref{Example2}), the geodesic convexity of the resulting flow would vanish as $n $  gets larger.

The previous discussion shows that the flow of $\jkl$ in the usual Wasserstein sense does not produce a flow with good  convergence properties that at the same time is cheap to discretize (robustly in $n$). This motivates considering the gradient flow of $\jkl$ with respect to the Wasserstein distance induced by a certain constant metric $g$. Indeed, inspired by Proposition \ref{example}, let us consider the constant metric tensor
\[  G:= L^\alpha. \]
 Since the metric tensor is constant, in particular its induced volume form $vol_g$ is proportional to the Lebesgue measure and hence we can write  
\[  d\mu^y(u) \propto  \exp \left( -\phi(u;y) - \frac{1}{2} \langle L^\alpha  u , u \rangle  \right) d vol_g(u) , \quad u \in  U .   \]
On the other hand, from the discussion in Section \ref{ssec:pdenormal} we know that the densities of the stochastic process 
\[ dX_t = - \nabla_g( \phi(X_t;y) + \frac{1}{2}\langle L^\alpha X_t, X_t \rangle   )dt + \sqrt{2}dB_t^g \]
correspond to to the gradient flow of the energy $\jkl$ with respect to the Wasserstein distance induced by the metric $g$, where $B^g$ is a Brownian motion on $(\R^m, g)$. This diffusion can be rewritten in terms of the standard Euclidean gradient $\nabla$ and Brownian motion $B$ as
\begin{equation}
 dX_t = -( X_t +L^{-\alpha}\nabla \phi(X_t; y)  )dt + \sqrt{2L^{-\alpha}}dB_t,  
 \label{pCNDiffu}
\end{equation}
after noticing that
\[ \nabla_g = G^{-1}  \nabla, \quad  B^g = \sqrt{G^{-1}} B,\]
where for the second identity we have used the fact that $G$ is constant.  How convex is the energy $\jkl$ with respect to the Wasserstein distance induced by $g$?  Since the metric tensor $G$ is constant it follows that
\[  \lambda_G := \inf_{x \in \R^m} \Lambda_{\min} \Bigl( G^{-1/2} \bigl( \text{Hess}\, F \bigr)G^{-1/2}  \Bigr), \]
where $F(u):= \phi(u;y) + \frac{1}{2}\langle L^\alpha u , u \rangle $. Finally, due to the convexity of $\phi(u;y)$ we deduce that 
\[  \lambda_G \geq \Lambda_{\min} \Bigl( G^{-1/2}  L^\alpha G^{-1/2}  \Bigr) = 1 . \]

We notice that in \eqref{pCNDiffu} $L$ appears as $L^{-\alpha}$. This is a fundamental difference from \eqref{difubasico} (where $L$ appears as $L^\alpha$) with computational advantages, given that the eigenvalues of $L$ grow towards infinity.

\begin{remark}
	A carefully designed discretization of \eqref{pCNDiffu} induces the so called Langevin pCN proposal for MCMC computing (see \cite{cotter2013mcmc}). 
\end{remark}

\begin{rem}
	In the above we have considered a probit model for the likelihood function. The ideas generalize straightforwardly to other settings, notably the logistic model
	\begin{equation}\label{logisticeq}
	\phi(u; y) := - \sum_{j \in \mathcal{Z}}\log\bigl( \sigma(y_j  u_j ; \gamma) \bigr) , \quad u \in U,
	\end{equation}
	where
	$$\sigma(t; \gamma):= \frac{1}{1+e^{-t/\gamma}}.$$
	The convexity of $\phi$ for the logistic model \eqref{logisticeq} can be established by direct computation of the second derivative of $\sigma.$
\end{rem}

\nc

%
%
%
%
%

\subsection{Ginzburg-Landau Model}\label{sssec:glandau}
We now present the Ginzburg-Landau model for semi-supervised learning. This model will provide us with an example of a functional $\jkl$ whose geodesic convexity with respect to Wasserstein distance is not positive (and hence one can not deduce geometric inequalities describing the rate of convergence towards the posterior), but for which one can obtain a positive spectral gap giving the rate of convergence of the flow of Dirichlet energy in the $L^2$ sense.

Let
$$W_\epsilon(t): = \frac{1}{4\epsilon}(t^2 - 1)^2, \quad t \in \R.$$
We consider the following Bayesian model.

\textbf{Prior:}
\begin{equation}\label{priorgl}
\frac{d\pi}{dv}(u) \propto \exp\Bigl( -\frac{1}{2}\langle u,L^\alpha u\rangle - \sum_{j\in \mathcal{Z}} W_\epsilon\bigl(u_j \bigr)  \Bigr) =: \exp\bigl(-\Psi(u)\bigr), \quad u\in U.
\end{equation}

\textbf{Likelihood function:}
For $j\in \mathcal{Z},$ 
$$y_j =  u_j + \eta_j, \quad \quad \eta_j \sim N(0,\gamma^2).$$
This leads to the following negative log-density function:
\begin{equation}\label{likelihoodgl}
\phi(u;y) =\frac{1}{2\gamma^2} \sum_{j\in \mathcal{Z}} |y_j - u_j|^2.
\end{equation}

\textbf{Posterior distribution:} Combining the prior and the likelihood via Bayes' formula gives the posterior distribution
$$\frac{d\mu^y}{d\pi}(u) = \exp\bigl( -\phi(u;y) \bigr); \quad \frac{d\mu^y}{du}(u) \propto \exp\bigl( -\Psi(u) - \phi(u;y)\bigr), \quad u \in U, $$
where $\Psi$ is given by \eqref{priorgl}, and $\phi$ is given by \eqref{likelihoodgl}.

For this model, the negative prior log-density $\Psi$ is not convex, and Wasserstein $\lambda$-geodesic convexity of the functional $\dkl(\cdot\| \pi)$ only holds for negative $\lambda.$  In particular, it is not possible to deduce exponential decay taking the Wasserstein flow point of view. However, in the $L^2$/Dirichlet energy setting we can still show exponential convergence towards the posterior $\mu^y$. Indeed,
because the negative log-likelihood of $\mu^y$ satisfies:
\[ \frac{|\nabla \Psi(u) + \nabla \phi(u;y) |^2}{2} - \Delta\Psi(u) - \Delta\phi(u;y) \rightarrow \infty, \quad \text{ as } |u| \rightarrow \infty, \]
there exists some $\lambda>0$ for which $\mu^y$ has a Poincar\'e inequality with constant $\lambda$ (see Chapter 4.5 in \cite{pavliotis2014stochastic}). In this example we can say more, and in particular we are able to find a Poincar\'e constant that depends explicitly on $\veps$,  the smallest non-trivial eigenvalue of $L,$ and $k:=|\mathcal{Z}|$. 
\nc

Let $\psi(u):=\sum_{j \in \mathcal{Z}} W_\veps(u_j)   $ and let  $\psi_c$ be its convex envelope, i.e. let $\psi_c$ be the largest convex function that is below $\psi$. It is straightforward to show that $\psi_c(0) = 0$ and that
\[ \inf_{u \in \R^n} \{\exp\left( \psi_c(u) - \psi(u))\right \} =\exp\left(-\frac{k}{4\veps}\right). \] 
Consider now the probability measure $\mu_c$ with Lebesgue density
\[ d \mu_c(u) = \frac{1}{Z_c} \exp \left( - \frac{1}{2}\langle L^\alpha u,u \rangle - \phi(u;y) - \psi_c(u) \right) d u,   \]
and define $\lambda_2$ and $\lambda_{2,c}$ as in Remark \ref{Eigangap} using $\mu^y$ and $\mu_c$ instead of $\mu$. For any given $f\in L^2(\mu)$ we then have
\[ \frac{\int | \nabla f|^2 d \mu }{\int |f-f_\mu|^2 d\mu}  \geq \frac{\int | \nabla f|^2 d \mu }{\int |f-f_{\mu_c}|^2 d\mu} \geq \exp\left(-\frac{k}{\veps} \right) \frac{\int | \nabla f|^2 d \mu_c }{\int |f-f_{\mu_c}|^2 d\mu_c}, \]
where the first inequality follows from the fact that $f_\mu = \argmin_{a \in \R} \int| f - a|^2 d \mu $ and the second inequality follows directly from the fact that $  0 \geq \psi_c - \psi \geq -\frac{k}{\veps} $ . It follows that
\begin{align*}
\lambda_2 &\ge \exp\left(-\frac{k}{\veps}\right) \, \lambda_{2,c} \geq  \exp\left(-\frac{k}{\veps}\right) (\Lambda_{\min}(L))^\alpha .
\end{align*}
where the last inequality follows from the fact that the negative log-likelihood of $\mu^y$ satisfies the Bakry-Emery condition with constant $\Lambda_{\min}(L)$ (see Chapter 4.5 in \cite{pavliotis2014stochastic}). Clearly, the Poincar\'e constant above is very large for small $\epsilon$ or for large $k$ (number of labeled data points).  We also notice that the cost of discretization of the diffusion associated to this flow increases with $n$ (as in Section \ref{sssec:probit}).

\begin{remark}
A similar analysis can be carried out now using the constant metric  
\[G:=L^\alpha.\]
More precisely, consider the flow of the Dirichlet energy
\[D^\mu_g(f):=\int \lvert \nabla_g f \rvert_g^2 d \mu(u) \]
with respect to $L^2(\mu)$. How convex is this functional? For every $f \in U$ we have
\[ \frac{\int | \nabla_g f|_g^2 d \mu }{\int |f-f_\mu|^2 d\mu}  \geq \frac{\int | \nabla_g f|_g^2 d \mu }{\int |f-f_{\mu_c}|^2 d\mu} \geq \exp\left(-\frac{k}{\veps} \right) \frac{\int | \nabla_g f|_g^2 d \mu_c }{\int |f-f_{\mu_c}|^2 d\mu_c}, \]
from where it follows that
\[ \lambda_2^g  \geq \exp\left( - \frac{k}{\veps} \right). \]
A similar remark to the one at the end of section \ref{sssec:probit} regarding the dependence in $L$ of the resulting diffusion applies here as well.
\end{remark}

\section{Conclusions and Future Work}\label{sec:conclusions}
The main contribution of this paper is to explore three variational formulations of the Bayesian update and their associated gradient flows. We have shown that, for each of the three variational formulations, the geodesic convexity of the objective functionals gives a bound on the rate of convergence of the flows to the posterior. As an application of the theory, we have suggested a criterion for the optimal choice of metric in Riemannian MCMC schemes.
We summarize below some additional outcomes and directions for further work.
\begin{itemize}
\item We bring attention to different variational formulations of the Bayesian update. These formulations have the potential of extending the theory of Bayesian inverse problems in function spaces, in particular in cases with infinite dimensional, non-additive, and non-Gaussian observation noise. Moreover, they suggest numerical approximations to the posterior by restricting the space of allowed measures in the minimization, by discretization of the associated gradient flows, or by sampling via simulation of the associated diffusion.
\item The variational framework considered in this paper provides a natural setting for the study of robustness of Bayesian models, and for the analysis of convergence of discrete to continuum Bayesian models. Indeed, the authors \cite{trillos2017continuum}, \cite{trillos2017consistency} have recently established  the consistency of Bayesian semi-supervised learning in the regime with fixed number of labeled data points and growing number of unlabeled data. The analysis relies on the variational formulation based on Kullback-Leibler prior penalization in equation \eqref{energyminimizationforlogit}.
\item The results in the paper give new understanding of the ubiquity of Kullback-Leibler penalizations in sampling methodology. In practice Kullback-Leibler is often used for computational and analytical tractability. The results presented in section \ref{ssec:PDEdiffusion} show that Kullback-Leibler prior penalization leads to a heat-type flow and, therefore, to an easily discretized diffusion process. On the other hand,  $\chi^2$ prior penalization leads to a nonlinear diffusion process.  
\end{itemize}

\subsection*{Acknowledgments}
We are thankful to Mat\'ias Delgadino for pointing to us the reference \cite{japoneses2011displacement} while participating in the CNA Ki-net workshop ``Dynamics and Geometry from High Dimensional Data'' that took place at Carnegie Mellon University in March 2017. We are also thankful to Sayan Mukherjee for the reference \cite{duke}. Finally, we thank the anonymous editor and referees for their immense help in improving the readability of our manuscript.

  \bibliographystyle{ba}
\bibliography{isbib}

\appendix
\section{Proof of Theorem \ref{theoremrmcmc}} 

Notice that the measure $\mu$ can be rewritten as
\[ d \mu(x) \propto \exp \biggl(- F(x) -  \log\Bigl( \sqrt{\det(G(x))} \,\, \Bigr)  \biggr) d vol_g(x) =\exp(-F_g(x) ) d vol_g(x), \]
where
\[F_g(x):= F(x) + \log\Bigl(\sqrt{\det (G)}\Bigr), \quad x \in \R^m. \]
From Proposition \ref{propositionkl}  the sharp constant $\lambda$ for which $\dkl(\cdot \| \mu)$ is $\lambda$-geodesically convex in the $g$-Wasserstein space is given by 
\[ \inf_{x \in \R^m} \min_{v: g(v,v) =1 } \text{Ric}_g(v,v) + \text{Hess}_g F_g(v,v),  \]
where $\text{Ric}_g$ and $\text{Hess}_g$ stand for the Ricci curvature and Hessian in the $g$-metric. To establish the proposition, it suffices to show that for any given $x \in \R^m$,  $\min_{v: g(v,v) =1 } \text{Ric}_g(v,v) + \text{Hess}_g F_g$ is equal to the smallest eigenvalue of 
the matrix $B+ \text{Hess}\, F- C$.

Let us start by recalling that the $g$-Hessian of a $C^2$ function  $I $, denoted by $\text{Hess}_g I$, is the symmetric $(2,0)$-tensor satisfying
\[ \text{Hess}_g I ( v , v  ) = \frac{d^2}{dt^2} I\bigl(\gamma(t)\bigr)\big|_{t=0},   \]
for every $v \in \R^m$ and every constant speed $g$-geodesic curve with $\gamma(0)=x$ and $\dot{\gamma}(0)=v$.
It is convenient to rewrite $ \text{Hess}_g I  (v , v  )$ in terms of the Christoffel symbols of the metric $g$, the Euclidean inner product, and the regular (Euclidean) gradient and Hessian of the function $I$. Let $\gamma: (-\veps, \veps ) \rightarrow \R^m$ be a constant speed $g$-geodesic with $\gamma(0)=x$, $\dot{\gamma}(0)=v$. We can then write:
\begin{align}
\begin{split}
\frac{d^2}{dt^2} I\bigl(\gamma(t)\bigr) &= \frac{d}{dt} \Bigl( \bigl\langle \nabla I (\gamma(t)) , \dot{\gamma}(t) \bigr\rangle \Bigr)
\\&= \Bigl\langle \text{Hess}\, I\bigl(\gamma(t)\bigr) \dot{\gamma}(t)   , \dot{\gamma}(t)  \Bigr\rangle + \Bigl\langle \nabla I (\gamma(t)) , \ddot{\gamma}(t) \Bigr\rangle,
\end{split}
\label{HessianComp}
\end{align}
where $\nabla I$ and $\text{Hess}\, I$ are the usual gradient and Hessian matrix of $I,$ respectively. The acceleration of the curve $\gamma$ can be written in terms of the Christoffel symbols.  Namely, if we write $\gamma$ in coordinates as
\[\gamma(t) = \bigl(\gamma_1(t), \dots, \gamma_m(t)\bigr), \quad t \in (-\veps, \veps),\]
the following system of second order ODEs holds:
\begin{equation*}
\ddot{\gamma}_l(t) = -\sum_{ij}\Gamma_{ij}^l \dot{\gamma}_i(t)\dot{\gamma}_j(t), \quad l=1, \dots, m.
\end{equation*} 
Plugging the geodesic equations back into \eqref{HessianComp}, and setting $t=0$, it follows that
\[  \frac{d^2}{dt^2}  I \bigl(\gamma(t)\bigr) |_{t=0} =  \Bigl\langle \bigl(\text{Hess}\, I -A(I)\bigr) v  , v \Bigr\rangle, \]
where $A(I)$ is the matrix with coordinates
\[ A(I)_{ij}=  \Gamma_{ij}^l \frac{\partial I}{\partial x_l} . \]
Hence, 
\begin{equation}
  \text{Hess}_g \, I (v , v ) =  \Bigl\langle \bigl(\text{Hess} I  - A(I)\bigr) v  , v \Bigr\rangle, \quad \forall v \in \R^m .
  \label{Hessg} 
\end{equation}

Taking $I:= F_g $ we can write the coordinate $ij$ of the matrix $ \text{Hess}\, I  - A(I)$ as
\[   \frac{\partial^2F}{\partial x_i x_j}   +  \frac{\partial^2}{\partial x_i x_j}  \log\left(\sqrt{\det(G)}\right) - \Gamma_{ij}^l\frac{\partial F}{\partial x_l } - \Gamma_{ij}^l \frac{\partial }{\partial x_l} \log\left(\sqrt{\det(G)}\, \right).   \] 
Using now the fact that the $g$-divergence of the vector field $\frac{\partial }{\partial x_l} $ can be written in terms of the Christoffel symbols as
\[  \frac{\partial}{\partial x_l} \log\Bigl(\sqrt{\det(G)}\Bigr) = \Gamma_{l k}^k,  \]
we deduce that
\[\bigl(\text{Hess}\, I  - A(I)\bigr)_{ij} = \frac{\partial^2F}{\partial x_i  \partial  x_j}   +  \frac{\partial \Gamma_{il}^l}{ \partial x_j}   - \Gamma_{ij}^l\frac{\partial F}{\partial x_l } - \Gamma_{ij}^l \Gamma_{lk}^k . \]
On the other hand, the Ricci curvature $\text{Ric}_g (v)$ can be written in terms of the Christoffel symbols and its derivatives (alternatively in terms of the metric and its first and second order derivatives) as
\begin{equation}
\text{Ric}_g(v,v) = \langle R v  ,v \rangle, \quad v \in \R^m,
\label{Riccig}
\end{equation}   
where $R$ is the (symmetric) matrix with entries
\[ R_{ij} = \frac{\partial \Gamma_{ij}^l}{\partial x_l} - \frac{\partial \Gamma_{il}^l }{\partial x_j}  + \Gamma_{ij}^l \Gamma_{kl}^k - \Gamma_{il}^k \Gamma_{jk}^l, \]
see \citet{docarmo1992riemannian} for details. After some cancellations using the symmetry of the symbols,  we obtain that
\[ \bigl( R + \text{Hess}\, I - A(I) \bigr)_{ij} =  \frac{\partial  \Gamma_{ij}^l }{\partial x_l} - \Gamma_{il}^k \Gamma_{jk}^l + \frac{\partial ^2 F}{\partial x_i \partial x_j} - \Gamma_{ij}^l \frac{\partial F}{\partial x_l},   \]
and so
\[ R + \text{Hess}\, I - A(I) = B + \text{Hess}\, F - C.\]
Using \eqref{Hessg} and \eqref{Riccig} we deduce that
\[ \text{Ric}_g (v,v) + \text{Hess}_g F_g(v,v) =  \bigl\langle (B+ \text{Hess}\, F  -  C ) v , v \bigr\rangle, \quad \forall v \in \R^m.   \]
Therefore the variational problem (for every fixed $x \in \R^m$)
\[ \min_{v : g(v,v)=1} \text{Ric}_g (v,v) + \text{Hess}_g F(v,v)   \]
can be rewritten, applying the change of variables $w:= G^{1/2}v,$ as
\[  \min_{v : \langle w , w \rangle=1} \Bigl\langle   G^{-1/2}\bigl(B +  \text{Hess}\, F - C \bigr) G^{-1/2} w, w \Bigr\rangle.     \]
In turn this coincides with 
$$\Lambda_{\min}\Bigl( G^{-1/2}\bigl(B +  \text{Hess}\, F - C \bigr) G^{-1/2}  \Bigr),$$ i.e., the smallest eigenvalue of the matrix 
\[ G^{-1/2}\bigl(B +  \text{Hess}\, F - C \bigr) G^{-1/2}.\] 
This concludes the proof of the first part of the theorem. 

Now we show the scaling property \eqref{homogeneity}.  Let $\widetilde{G} = aG$. By definition
\begin{align*}
\lambda_{\widetilde{G}} = \inf_{x \in \R^m} \Lambda_{\min} \Bigl( \widetilde{G}^{-1/2} \bigl(\widetilde{B} + \text{Hess}\,F - \tilde{C} \bigr)\widetilde{G}^{-1/2}  \Bigl),
\end{align*}
where $\widetilde{B}$ and $\widetilde{C}$ are defined as in \eqref{def:B} and \eqref{def:C} but in terms of the metric $\tilde{G}.$ From the expression \eqref{Christoffelsymbols} for the Christoffel symbols, it follows that they are invariant under rescaling of the metric and, since $B, \, \widetilde{B}$ and $C, \,\widetilde{C}$ depend on the metric only through the symbols, we deduce that  $\widetilde{B} = B $, $\widetilde{C} = C.$
Therefore,
\begin{align*}
\lambda_{\widetilde{G}} &=\frac{1}{a} \inf_{x \in \R^m} \Lambda_{\min} \Bigl( {G}^{-1/2} \bigl(B + \text{Hess}\,(F) - C \bigr) {G}^{-1/2}  \Bigr) \\
&= \frac{1}{a} \lambda_G. 
\end{align*}
  \qed

\end{document}